\documentclass[a4paper, 12pt]{amsart}

\usepackage{amssymb}
\usepackage{amsthm}
\usepackage{amsmath}
\usepackage{mathrsfs}
\usepackage{comment}

\theoremstyle{plain}
\newtheorem{thm}{Theorem}[section]		
\newtheorem{prop}[thm]{Proposition}
\newtheorem{cor}[thm]{Corollary}
\newtheorem{lem}[thm]{Lemma}
\newtheorem{clm}[thm]{Sublemma}
	
\theoremstyle{definition}
\newtheorem{df}{Definition}[section]
\newtheorem{exam}{Example}[section]

\theoremstyle{remark}
\newtheorem{rmk}{Remark}[section]
\newtheorem*{ac}{Acknowledgements}

\newcommand{\To}{\Longrightarrow}

\newcommand{\nn}{\mathbb{N}}
\newcommand{\zz}{\mathbb{Z}}
\newcommand{\qq}{\mathbb{Q}}
\newcommand{\rr}{\mathbb{R}}

\DeclareMathOperator{\card}{card}

\DeclareMathOperator{\cl}{CL}

\DeclareMathOperator{\pc}{PC}
\DeclareMathOperator{\nai}{INT}
\DeclareMathOperator{\tpc}{TPC}

\DeclareMathOperator{\kpc}{KPC}
\DeclareMathOperator{\cdim}{Cdim}
\DeclareMathOperator{\cov}{cov}
\DeclareMathOperator{\dom}{dom}
\makeatletter
\@addtoreset{equation}{section}

\makeatother

\begin{document}

\title[A characterization of metric subspaces]
{A characterization of metric subspaces of full Assouad dimension}
\author[Yoshito Ishiki]
{Yoshito Ishiki}
\address[Yoshito Ishiki]
{\endgraf
Graduate School of Pure and Applied Sciences
\endgraf
University of Tsukuba
\endgraf
Tennodai 1-1-1, Tsukuba, Ibaraki, 305-8571, Japan}
\email{ishiki@math.tsukuba.ac.jp}

\date{\today}
\thanks{The author was supported by JSPS KAKENHI Grant Number 18J21300.}

\begin{abstract}
We introduce the notion of tiling spaces for metric spaces.
The class of tiling spaces contains 
the Euclidean spaces, 
the middle-third Cantor set, 
and various self-similar spaces appearing  in fractal geometry. 
For  doubling tiling spaces, 
we characterize  metric  subspaces whose Assouad dimension coincides with that of the whole space. 
\end{abstract}
\maketitle

\section{Introduction}
Fraser and  Yu \cite{FY} provided a characterization of
  subsets of
 a Euclidean space  whose Assouad dimension coincides with that of the whole space. 
Namely, 
they proved in \cite{FY} that 
for every subset  $F$ of  the $N$-dimensional Euclidean space $\rr^N$,  
the following are equivalent:
\begin{enumerate}
\item $F$ asymptotically contains arbitrary large arithmetic patches; 
\item $F$ satisfies the asymptotic Steinhaus property;
\item $\dim_AF=N$, where $\dim_A$ stands for the Assouad dimension;
\item $\cdim_AF=N$, where $\cdim_A$ stands for the conformal Assouad dimension;
\item $F$ has a weak tangent with non-empty interior;
\item the closed unit ball $B(0,1)$ in $\rr^N$ is a weak tangent to $F$. 
\end{enumerate}
They used this characterization to study the problem 
of whether specific subsets of Euclidean spaces related to number theory, 
such as the products of the set of all prime numbers, 
asymptotically contain higher dimensional arithmetic progression. 
Their results in \cite{FY} are related to the Erd\H{o}s--Tur\'an conjecture. 

The purpose of this paper is to generalize 
the Fraser--Yu characterization to general metric spaces. 
In the proof of the Fraser--Yu characterization, 
they essentially used the fact that Euclidean spaces are tiled by congruent cubes. 
From this point of view, 
we introduce the  notion of tiling spaces for metric spaces. 
The class of tiling spaces contains 
the Euclidean spaces, 
the middle-third Cantor set, 
and various self-similar spaces appearing in fractal geometry. 
For  doubling tiling spaces,
 we characterize 
  metric  subspaces  whose Assouad dimension coincides with that of the whole space.

We first define covering pairs as follows:
For a set $X$, 
we denote by $\cov(X)$ the set of all coverings of $X$. 
We call a map $P$ from $\nn$ or $\zz$ to $\cov(X)$ a \emph{covering structure on $X$}. 
We denote by $P_n$ the value of $P$ at $n$. 
For $T\in P_{n}$ and $k\in \nn$, 
we put $[T]_{k}=\{\, A\in P_{n+k}\mid A\subset T\, \}$.

We call a pair $(X, P)$ of a set $X$ and a covering structure $P$ on $X$   a \emph{covering pair}. 
We denote by $\dom(P)$ the domain of the map $P$. 
Note that $\dom(P)$ coincides with either  $\nn$ or $\zz$. 

\begin{df}[Tiling set]
Let $(X, P)$ be a covering pair. 
 For $N\in \nn$, we say that $(X, P)$ is an \emph{$N$-tiling set} if 
it satisfies the following: 
\begin{enumerate} 
\item[(S1)] for each pair  $n,m\in \dom(P)$ with $n<m$, and for each $A\in P_n$, 
we have $\card([A]_{m-n})=N^{m-n}$ 
and $A=\bigcup[A]_{m-n}$, 
where the symbol $\card$ stands for the cardinality;
\item[(S2)] for each $n\in \dom(P)$, 
and for each pair  $A, B\in P_n$, 
there exist $m\in \dom(P)$ and $C\in P_m$ such that $A\cup B\subset C$
and $m<n$.
%%%%%%%%%%%%%%%%MO%%%%%%%%%%%%%
\item[(S3)] for each $n\in \dom(P)$, for all $l, m\in \nn$,  and for each $A\in P_n$, we have 
\[
[A]_{m+l}=\bigcup_{T\in [A]_{m}}[T]_{l}. 
\]
%%%%%%%%%%%%%%%%%%%
 \end{enumerate}
 We say that $(X, P)$ is a \emph{tiling set} if 
 it is an $N$-tiling set for some $N$. 
 For a tiling set $(X,P)$, 
we say that a subset $T$ of $ X$ is a \emph{tile} of $(X, P)$ 
if there exists $n\in \dom(P)$ such that $T\in P_n$.

\end{df}
We next specialize the notion of tiling sets for metric spaces. 
For a metric space $X$ and for a subset $A$ of $X$, 
we denote by $\delta(A)$ the diameter of $A$. 
For $p\in X$ and $r\in (0, \infty)$, 
we  denote by $U(p, r)$ the open ball centered at $p$ with radius $r$. 
For $h\in (0, \infty)$, 
and for a metric space $X$ with metric $d_X$,  
we denote by $hX$ the metric space $(X, hd_X)$. 
\begin{df}[Tiling space]
Let $N\in \nn$ and $s\in (0,\infty)$. 
Let $X$ be a metric space. 
Let $(X, P)$ be an $N$-tiling set. 
We say that 
$(X, P)$ is an $(N, s)$-\emph{pre-tiling space} if 
it satisfies the following:
\begin{enumerate}
\item[(T1)] there exist $D_1>0$ and $D_2>0$ such that 
for each $n\in \dom(P)$ 
and for each $A\in P_n$, we have $D_1\le \delta(A)/s^n\le D_2$;  
\item[(T2)] there exists $E>0$ such that for each $n\in \dom(X)$ 
and for each $A\in P_n$, 
there exists a point $p_A\in A$ with $U(p_A, Es^m)\subset A$. 
 \end{enumerate}
Furthermore, 
we say that an $(N, s)$-pre-tiling space $(X,P)$ is an \emph{$(N, s)$-tiling space} if 
it satisfies: 
\begin{enumerate}
\item[(U)] for each countable sequence  $\{A_i\}_{i\in \nn}$ of tiles of $(X, P)$, 
there exists a subsequence $\{A_{\phi(i)}\}_{i\in \nn}$ such that 
$\{(\delta(A_{\phi(i)}))^{-1}A_{\phi(i)}\}_{i\in \nn}$  converges to $(\delta(T))^{-1}T$ for some tile $T$ of $(X, P)$ 
  in the sense of Gromov--Hausdorff. 
\end{enumerate}
We say that $(X, P)$ is a \emph{tiling} (resp. \emph{pre-tiling}) space  
if it is an $(N, s)$-tiling (resp.~$(N, s)$-pre-tiling) space for some $N$ and $s$. 
 \end{df}
 We say that two metric spaces $X$ and $Y$ are \emph{similar} 
if there exists $h\in (0, \infty)$ with $d_{GH}(hX, Y)=0$, where $d_{GH}$ is  the Gromov--Hausdorff distance. 
Similarity is  an equivalence relation on metric spaces. 
Note that every  metric space is similar to  its completion. 
 
Let $X$ be a metric space. 
Let $(X,P)$ be a tiling set.  
If the similarity classes of the tiles of $(X, P)$ is finite,
then the condition (U) is satisfied.  
Thus the condition (U) is considered as a generalization  of the finiteness of the similarity classes of tiles. 
There exists a pre-tiling space failing the condition (U) whose  tiles have infinite similarity classes (see Example \ref{exam:seqCan}). 
There exists a tiling space whose tiles have infinite similarity classes (see Example \ref{exam:infinite}). 
The $p$-adic numbers (Example \ref{exam:padic}) are non-Euclidean 
examples of tiling spaces.

 To state our main result, we use the notion of pseudo-cones introduced by 
 the author \cite{I}. 
Let $X$ be a metric space. 
Let $\{A_i\}_{i\in \nn}$ be a sequence of subsets of $X$,  
and let  $\{u_{i}\}_{i\in \nn}$ be a sequence in $(0, \infty)$. 
We say that a metric space $P$ is a \emph{pseudo-cone of $X$ approximated by 
$(\{A_i\}_{i\in \nn}, \{u_i\}_{i\in \nn})$} if 
$\lim_{i\to \infty}d_{GH}(u_iA_i, P)=0$.
We denote by $\pc(X)$ the class of all pseudo-cones,  
and by $\kpc(X)$ the class of all pseudo-cones approximated by 
a pair of  a sequence  $\{A_i\}_{i\in \nn}$ of compact sets of  $X$ and a sequence $\{u_i\}_{i\in \nn}$ in $(0, \infty)$. 
Let $(X,P)$ be a pre-tiling space, and let $F$ be a subset of $ X$. 
We also denote by $\tpc(F)$ the class of all pseudo-cones approximated by 
$(\{A_i\cap F\}_{i\in \nn}, \{u_i\}_{i\in \nn})$, 
where $\{A_i\}_{i\in\nn}$ is a sequence of tiles of $(X, P)$ and $\{u_i\}_{i\in\nn}$ is a sequence in $(0, \infty)$.
We emphasize that, in this paper, 
we use the Gromov--Hausdorff distance between not only compact metric spaces but also non-compact ones. 
Thus $d_{GH}$ is not necessarily a metric.

Let $(X,P)$ be a tiling set. 
Let $A$ be a tile of $(X,P)$. 
We say that a subset $F$ of $X$ satisfies  
the \emph{asymptotic Steinhaus property for $A$} if
for each $\epsilon>0$ and for each finite subset $S$ of $A$,
 there exist a finite subset $T$ of $ F$, 
 and 
 $\delta\in (0, \infty)$ such that  $d_{GH}(T, \delta\cdot S)\le \delta\cdot \epsilon$.

Our main result of this paper is the following characterization:
\begin{thm}\label{thm:tiling}
Let $(X, P)$ be a doubling tiling space. 
Then for every subset $F$ of $X$
the following are equivalent:
\begin{enumerate}
\item $\dim_AF=\dim_AX$; \label{item:dimful}
\item there exists a tile $A$ of $(X, P)$ such that $A\in \pc(F)$; \label{item:pc}
\item there exists a tile $A$ of $(X, P)$ such that  $A\in \tpc(F)$; \label{item:tpc}
\item there exists a tile $A$ of $(X, P)$ such that  $A\in \kpc(F)$;  \label{item:kpc}
\item there exists a tile $A$ of $(X, P)$ such that $F$ satisfies the asymptotic Steinhaus property for $A$. \label{item:asp}
\end{enumerate}
\end{thm}
In Theorem \ref{thm:tiling}, 
we need the assumption of  the doubling property for $X$, 
which is  equivalent to the finiteness of the Assouad dimension.
The definition of the doubling property can be seen in Section \ref{sec:pre}. 
There exists a  tiling space that is not doubling (see Example \ref{exam:nondb}). 

%%%%%%%%%%%%%%%%modify%%%%%%%%%%%%
\begin{rmk}
Let $(X, P)$ be a tiling space. 
If $X$ is doubling, then for every tile $T\in P$ we have 
$\dim_AT=\dim_AX$ (see Corollary \ref{cor:dim}). 
If $X$ is not doubling, then the equality does not necessarily hold. 
For example, the tiling space constructed in Example \ref{exam:nondb} 
has infinite Assouad dimension, and possesses a tile of  finite  Assouad dimension. 
\end{rmk}
%%%%%%%%%%%%%%%%%%%%%%%%%%%%%%%

If a tiling space $(X, P)$ satisfies the assumption that  the conformal dimensions of 
all the tiles of $(X, P)$ and $X$ are equal to 
$\dim_AX$, 
 then the condition that  $\cdim_AF=\dim_AX$ is equivalent to the conditions (1)--(5) stated in Theorem \ref{thm:tiling}. 
The assumption mentioned  above seems to be  quite strong.
Indeed, 
the author does not  know an  example satisfying the assumption except  the Euclidean spaces. 
We do not deal with  the conformal dimensions of tiling spaces. 

%%%%%%%%%%%%%%%MM
Attractors of  iterated function systems on metric spaces are studied as canonical examples of  fractals, and their Hausdorff dimensions are investigated (see e.g., \cite[Chapter 9]{Fal}, \cite{Sch1},  and \cite{Sch2}).  
%%%%%%%%%%%%%%%%%%%%%%%%%%%%%%%
\begin{df}\label{def:ifs}
Let $X$ be a complete metric space. 
For $L\in (0, \infty)$, a map $f:X\to X$ is said to be an \emph{$L$-similar transformation on $X$} if 
for all $x, y\in X$ we have $d_X(f(x), f(y))=Ld_X(x,y)$.  
Let $N\ge 2$ and let $s\in (0,1)$. 
We say that  $\mathcal{S}$  is an \emph{$(N, s)$-similar iterated function system on $X$} if 
$\mathcal{S}$ consists of $N$ many $s$-similar transformations on $X$, 
say $\mathcal{S}=\{S_i\}_{i=0}^{N-1}$. 
A non-empty subset $F$ of $X$ is said 
to be an \emph{attractor of the iterated function system $\mathcal{S}$} if 
$F$ is compact and  $F=\bigcup_{i=0}^{N-1}S_i(F)$. 
Since $X$ is complete, 
an attractor of $\mathcal{S}$ always uniquely exists (see \cite[Chapter 9]{Fal} for the Euclidean setting). 
We write $A_{\mathcal{S}}$ as the attractor of $\mathcal{S}$. 
We say that the system $\mathcal{S}$ satisfies the \emph{strong open set condition} 
if there exists an open set $V$ of $X$ such that 
\begin{enumerate}
\item[(O1)] $\bigcup_{i=0}^{N-1}S_i(V)\subset V$;
\item[(O2)] $\{S_i(V)\}_{i=0}^{N-1}$ are pairwise disjoint;
\item[(O3)] $V\cap A_{\mathcal{S}}\neq\emptyset$. 
\end{enumerate}
Let $W$ be the set of all words generated by $\{0, \dots, N-1\}$. 
For each word $w=w_0\cdots w_l$,
we write $S_w=S_{w_l}\circ \cdots \circ S_{w_0}$, 
where each  $w_i$ belongs to $\{0, \dots, N-1\}$. 
We define a map $P_\mathcal{S}:\nn\to \cov(A_{\mathcal{S}})$ by
\begin{equation}\label{eq:ifs}
(P_\mathcal{S})_n=\{\, S_w(A_{\mathcal{S}})\mid \text{$w\in W $ and $|w|=n$}\, \}, 
\end{equation}
where $|w|$ stands for the length of the word $w$. 
\end{df}
 Similar iterated function systems provide us a plenty of  tiling spaces. 

\begin{thm}\label{thm:frac}
For $N\ge 2$ and $s\in (0, 1)$, 
let $\mathcal{S}$ be an $(N, s)$-similar iterated function system 
on a complete metric space with the strong open set condition. 
Let $A_{\mathcal{S}}$ be the attractor of $\mathcal{S}$, 
and 
$P_\mathcal{S}$ the map defined by \emph{(\ref{eq:ifs})}. 
Then the pair $(A_{\mathcal{S}}, P_\mathcal{S})$ is an $(N, s)$-tiling space. 
\end{thm}
By Theorem \ref{thm:frac},  for instance, 
the middle-third Cantor set and the Sierpi\'nski gasket are 
tiling spaces for some suitable covering structures (see Subsection \ref{subs:attractor}), 
and we can apply Theorem \ref{thm:tiling} to them.

The organization of this paper is as follows: 
In Section \ref{sec:pre}, 
we review the definitions and basic properties of the Assouad dimension 
and the Gromov--Hausdorff distance. 
In Section \ref{sec:tiling}, 
we discuss basic properties of tiling spaces. 
In Section \ref{sec:Assouadtiling}, 
we prove Theorem \ref{thm:tiling}. 
In Section \ref{sec:exam},  
we show Theorem \ref{thm:frac},  
and  provide  tiling spaces induced from similar iterated function systems. 
In Section \ref{sec:counter},  
we exhibit counterexamples related to our characterization of tiling spaces. 
\begin{ac}
The author would like to thank Professor Koichi Nagano for his advice and constant encouragement. 
%%%%%%%%%modify%%%%%%%
The author also would like to thank the referee  for 
helpful comments,  and for  essential suggestions on the references. 
%%%%%%%%%%%%%%%%%%%
\end{ac}

\section{Preliminaries}\label{sec:pre}
In this paper, 
we denote by $\nn$ the set of all non-negative integers. 
\subsection{Metric spaces}
Let $X$ be a metric space. 
The symbol $d_X$ stands for the metric of $X$. 
We denote by $B(x,r)$ (resp.~$U(x,r)$) the closed (resp.~open) ball centered at $x$ with radius $r$. 
For a subset $A$ of $X$, and for $r\in (0, \infty)$, 
%%%%%%%%%%%%%%Modify%%%%%%%%%%%%%%
we  denote by 
$B(A, r)$ the closed ball centered at $A$ with radius $r$ defined by  
 $B(A, r)
 =\{\,x\in X\mid \inf_{a\in A}d(x, a)\le r\, \}$. 
%%%%%%%%%%%%%%Modify%%%%%%%%%%%%%%
For a subset $A$ of $X$, 
we set  $\alpha(A)=\inf\{\, d_X(x, y)\mid \text{$x, y\in A$ and $x\neq y$}\, \}$. 
We say that $A$ is a \emph{separated set} if there exists $r\in (0, \infty)$
with $\alpha(A)\ge r$.

Let $p\in [1, \infty]$.  
For two metric spaces $X$ and $Y$, 
we denote by $X\times_pY$  the product metric space of $X$ and $Y$ 
with the $\ell^p$-product metric  $d_{X\times_pY}$ defined by 
\[
d_{X\times_p Y}((a,b), (c,d))=
\begin{cases}
(d_X(a,c)^p+d_Y(b,d)^p)^{1/p} &\text{if $p\in [1, \infty)$}, \\
\max\{d_X(a,c),\ d_Y(b,d)\} & \text{if $p=\infty$}. 
\end{cases}
\]
\subsection{Assouad dimension}

For $N\in \nn_{\ge 1}$, 
we say that a metric space $X$ is \emph{$N$-doubling} 
if for every bounded set $S$ of $X$, 
there exists a subset $F$ of $X$ such that 
$S\subset B(F, \delta(S)/2)$ and $\card(F)\le N$, 
where $B(F, \delta(S)/2)$ is 
the closed ball centered at $F$ with radius $\delta(S)/2$. 
A metric space is said to be \emph{doubling} 
if it is $N$-doubling for some $N$.\par

Let $X$ be a metric space. 
For a bounded set $S$ of $X$ and $r\in (0, \infty)$, 
we denote by $\mathcal{B}_X(S,r)$ the minimum integer $N$ such that 
$S$ can be covered by at most $N$ bounded sets with diameter at most $r$. 
We denote by $\mathscr{A}(X)$ the set of all
$\beta\in (0, \infty)$ 
for which there exists $C\in (0,\infty)$ such that
for every bounded set $S$ of $X$ and for every $r\in (0, \infty)$
we have $\mathcal{B}_X(S, r)\le C(\delta(S)/r)^{\beta}$. 
We also denote by $\mathscr{C}(X)$ the set of  all $\gamma\in (0,\infty)$ such that 
there exists $C\in (0, \infty)$ such that for every bounded set 
$S$ of $X$ and  for every separated subset  $M$ of $S$, 
we have $\card(M)\le C(\delta(S)/\alpha(M))^{\gamma}$.

The \emph{Assouad dimension} $\dim_AX$ of a metric space $X$ is defined as 
$\inf (\mathscr{A}(X))$
 if $\mathscr{A}(X)$ is non-empty;
otherwise, $\dim_A(X)=\infty$.

By the definitions, we have the next two propositions. 
\begin{prop}\label{prop:Assouaddim}
For every metric space $X$, the following are equivalent:
\begin{enumerate}
 \item  $X$ is doubling;
\item $\mathscr{A}(X)$ is non-empty;
\item  $\mathscr{C}(X)$ is non-empty;  
\item  $\dim_AX<\infty$. 
\end{enumerate}
\end{prop}
\begin{prop}\label{prop:Assouad}
For every metric space $X$, we have 
\[
\dim_AX=\inf(\mathscr{C}(X)).
\] 
\end{prop}

\subsection{Gromov--Hausdorff distance}
For a metric space $Z$,  and  for subsets $A, B$ of $ Z$, 
we denote by $d_H(A,B;Z)$ the Hausdorff distance between $A$ and $B$ in $Z$. 
For two metric spaces $X$ and $Y$, 
the \emph{Gromov--Hausdorff distance} $d_{GH}(X,Y)$ between $X$ and $Y$ is defined as 
the infimum of  all values  $d_H(i(X), j(Y); Z)$, 
where $Z$ is a metric space and $i:X\to Z$ and $j:Y\to Z$ are isometric embeddings. 

%%%%%%%%%%%%%%modify%%%%%%%%%%%%
%%%%%%%%%%%%%%%%%%%%%%%%%%%

By the definition of the Gromov--Hausdorff distance, 
we have:
\begin{prop}\label{prop:distance}
Let $h\in (0, \infty)$. Let $X$ and $Y$ be metric spaces. 
Then $d_{GH}(hX, hY)=hd_{GH}(X,Y)$. 
\end{prop}

Let $X$ be a metric space. 
For $\epsilon\in (0, \infty)$, 
we define a function $d_X^{\epsilon}:X\times X\to [0, \infty)$ 
by $d_X^{\epsilon}(x, y)=(d_X(x, y))^{\epsilon}$. 
If $d_X^{\epsilon}$ is a metric, 
then we denote by $X^{\epsilon}$ the metric space $(X, d_X^{\epsilon})$. 
The metric  space $X^{\epsilon}$ is called a \emph{snowflake of $X$}. 
 Note that if $\epsilon\in (0,1)$, then $d_X^{\epsilon}$ is a metric.

\begin{prop}\label{prop:snow}
Let $\epsilon\in (0,1)$. Let $X$ and $Y$ be metric spaces. 
Then $d_{GH}(X^{\epsilon}, Y^{\epsilon})=d_{GH}(X,Y)^{\epsilon}$. 
\end{prop}
\begin{proof}
Let $Z$ be a metric space. 
For all subsets  $A, B\subset Z$, 
we have
\[
d_H(A, B; Z^{\epsilon})=d_H(A, B;Z)^{\epsilon}.
\] 
This leads to the proposition. 
\end{proof}

Let $X$ be a metric space. Let $\epsilon\in (0, \infty)$. 
We say that a subset $S$ of $X$ is an \emph{$\epsilon$-net} if 
$S$ is finite and $B(S, \epsilon)=X$, 
where $B(S, \epsilon)$ is the closed ball centered at $S$ with radius $r$. 

A metric space $X$ is said to be \emph{totally bounded} if 
for each $\epsilon\in (0, \infty)$ the space $X$ contains an $\epsilon$-net. 
A metric space $X$ is totally bounded 
if and only if 
$X$ is approximated by its finite subset in the sense of Gromov--Hausdorff. 

By the definitions of the total boundedness and $d_{GH}$, 
we have: 
\begin{prop}\label{prop:ttapprox}
Let $X$ be a totally bounded metric space, and $Y$ a metric space. 
If $d_{GH}(X, Y)\le \epsilon$, then there exists a finite subset $E$ of $Y$ such that $d_{GH}(X, E)\le 2\epsilon$. 
\end{prop}
%%%%%%%%%%%%%%%%%%%%%%%%%%%%%%%%
%%%%%%%%%%%%%%%%%%%%%%%%%%%%%%%%5
%%%%%%%%%%%%%%%%%%%%%%%%%%%%%%%%%%

\section{Properties of spaces with tiling structures}\label{sec:tiling}
%%%%%%%%%%%%%%%%%%%%%%%%%%%%%%%%%%%
%%%%%%%%%%%%%%%%%%%%%%%%%%%%%%%%%%%%
%%%%%%%%%%%%%%%%%%%%%%%%%%%%%%%%%%%%
We discuss basic properties of tiling sets, and  (pre-)tiling spaces. 
\begin{prop}\label{prop:uni}
Let $(X, P)$ be a tiling set. 
Then for every pair  $n, m\in \dom(P)$ with $n<m$, and for every $A\in P_m$, 
 there exists a unique $B\in P_n$ with $A\subset B$. 
\end{prop}
\begin{proof}
Let $(X, P)$ be an $N$-tiling set for some $N$. 
Suppose that there exist $B, C\in P_n$ with $B\neq C$ and $A\subset B\cap C$. By the condition (S2), 
there exist $k\in \dom(P)$ and  $D\in P_k$ such that $B\cup C\subset D$ and $k<n$.  
By the condition (S1), we have $\card([D]_{m-k})=N^{m-k}$. 
%%%%%%%%%%%%%%%%%%%%%%MODIFY%%%%%%%

Put $[D]_{n-k}=\{T_i\}_{i=1}^{N^{n-k}}$. 
We may assume that 
$T_1=B$ and $T_2=C$. 
Then the condition (S3) yields $[D]_{m-k}=\bigcup_{i=1}^{N^{n-k}}[T_i]_{m-n}$. 
By $A\subset B\cap C$,  we have $A\in [T_1]_{m-n}\cap [T_2]_{m-n}$, 
and hence 
\[
[D]_{m-k}=([T_1]_{m-n}\setminus\{A\})\cup \bigcup_{i=2}^{N^{n-k}}[T_i]_{m-n}. 
\]
Since  the condition (S1) implies that  for every $i\in \{1, \dots, N^{n-k}\}$ we have $\card([T_i]_{m-n})=N^{m-n}$,  we obtain
\begin{align*}
\card([D]_{m-k})&\le (\card([T_1]_{m-n})-1)+\sum_{i=2}^{N^{n-k}}\card([T_i]_{m-n})\\
&=N^{m-k}-1<N^{m-k}. 
\end{align*}
This is a contradiction. 
\end{proof}
%%%%%%%%%%%%%%%%%%%%%%%%%%%%%%5
By the conditions (S2) and (T1), 
we have the following propositions:
\begin{prop}\label{prop:bddtiling}
Let $(X, P)$ be a pre-tiling space. 
If $X$ is bounded, then $\dom(P)=\nn$ and  $P_0=\{X\}$. 
\end{prop}
\begin{proof}
By the condition (T1) and the boundedness of $X$, 
the tiling index $\dom(P)$ must be $\nn$. 
If $P_0$ would have  two elements, 
then this contradicts the condition (S2) and $\dom(P)=\nn$. 
Thus  $P_0=\{X\}$. 
\end{proof}
%%%%%%%%%%%%%%%%%%%%%%%%%%%%%%%%%%%%%
\begin{prop}\label{prop:nnzz}
Let $(X, P)$ be a pre-tiling space. 
The space $X$ is bounded 
if and only if $\dom(P)=\nn$. 
Equivalently, 
the space $X$ is unbounded 
if and only if $\dom(P)=\zz$. 
\end{prop}
\begin{proof}
By Proposition \ref{prop:bddtiling}, 
it suffices to show that if $\dom(P)=\nn$, 
then $X$ is bounded. 
This holds true by the condition (T1). 
 \end{proof}

%%%%%%%%%%%%%%%%%%%%%%%%%%
For a subset $A$ of a metric space $X$, 
we denote by $\nai(A)$ the interior of $A$ in $X$.  

%%%%%%%%%%%%%%%%%%%%%%%%%%%%%%%
\begin{lem}\label{lem:nai}
Let $(X, P)$ be a pre-tiling space. 
For every distinct pair   $A,B\in P_n$, 
we have $\nai(A)\cap \nai(B)=\emptyset$. 
\end{lem}
\begin{proof}
If
for some distinct $A,B\in P_n$ 
we have  $\nai(A)\cap \nai(B)\neq \emptyset$, 
then, 
by the condition (T1),  
there exist $k\in \dom(P)$  and $C\in P_k$ such that
$C\subset \nai(A)\cap \nai(B)$. 
This contradicts Proposition \ref{prop:uni}. 
\end{proof}
\begin{lem}\label{lem:totbdd}
Every bounded pre-tiling space is totally bounded. 
\end{lem}
\begin{proof}
Let $(X, P)$ be a bounded $(N,s)$-pre-tiling space, 
and 
let $D_2$ be a constant appearing in the condition (T1). 
Proposition \ref{prop:nnzz} implies $\dom(P)=\nn$.  
For each $n\in \nn$, and for each $A\in [X]_n$, 
take a point $q_A\in A$. 
By the conditions (S1) and  (T1), 
the set $\{\, q_A\in X\mid A\in [X]_n\, \}$ is a  $(D_2s^n)$-net of $X$. 
\end{proof}
%%%%%%%%%%%%%%%%%%%%%%%%%%%%%%%
Since a totally bounded complete metric space is compact, 
we have:
\begin{cor}\label{cor:bddcomp}
Every bounded complete pre-tiling space is compact. 
\end{cor}
%%%%%%%%%%%%%%%%%%%%%%%%%%%%%%%
Since a totally bounded metric space is separable, 
we obtain:
\begin{cor}\label{prop:bddsep}
Every bounded pre-tiling space is separable.  
\end{cor}
%%%%%%%%%%%%%%%%%%%%%%%%%%%%%%%
We next show the countability of tiling structures. 
\begin{prop}\label{prop:tilectb}
Let $(X, P)$ be a pre-tiling space. 
Then  each $P_n$ is a countable family. 
\end{prop}
\begin{proof}
By Propositions \ref{prop:bddtiling} and \ref{prop:nnzz}, 
we may assume that $X$ is unbounded. 
Take a sequence $\{T_i\}_{i\in \nn}$ of tiles of $(X, P)$ such that 
for each $i\in \nn$ 
we have $T_i\in P_{-i}$ and $T_i\subset T_{i+1}$. 
By the condition (S2) and Proposition \ref{prop:uni}, 
we have $X=\bigcup_{i\in \nn}T_i$.  
Then we obtain  $P_n=\bigcup_{-i\le n}[T_i]_{n+i}$. 
This shows the proposition. 
\end{proof}
By Propositions \ref{prop:bddsep} and \ref{prop:tilectb}, 
we obtain: 

%%%%%%%%%%%%%%%%%%%%%%%%%%%%%%%
\begin{cor}
Every pre-tiling space is separable. 
\end{cor}

%%%%%%%%%%%%%%%%%%%%%%%%%%%%%%55
%%%%%%%%%%%%%%%%%%%%%%%%%%%%%%%%5

Let $f:X\to Y$ be a map between metric spaces. 
Let  $L\in [1, \infty)$ and $\gamma\in (0,\infty)$. 
We say that $f$ is \emph{$(L, \gamma)$-homogeneously bi-H\"older} if  
for all $x, y\in X$ we have 
\[
L^{-1}d_{X}(x, y)^{\gamma}\le d_Y(f(x), f(y))\le Ld_X(x, y)^{\gamma}.
\] 
A \emph{homogeneously bi-H\"older map} means 
an $(L, \gamma)$-homogeneously bi-H\"older map 
for some $L$ and $\gamma$. 
A map $f:X\to Y$ is \emph{$L$-bi-Lipschitz} if
it is $(L, 1)$-homogeneously bi-H\"older. 
%%%%%%%%%%%%%%%%%%%%%%%%%%%%%%%%%%%%

%%%%%%%%%%%%%%%%Modify%%%%%%%%%%%%%%%
\begin{rmk}
Let $X$ be an ultrametric space. 
For every $\gamma\in (0, \infty)$, the function $d_X^{\gamma}$ is also 
an ultrametric. 
Therefore the identity map $id: X\to X^{\gamma}$ is 
$(1, \gamma)$-homogeneously bi-H\"older for any $\gamma\in (0, \infty)$. 
\end{rmk}
%%%%%%%%%%%%%%%%%%%%%%%%%%%%

\begin{lem}\label{lem:Holder}
Let $f:X\to Y$ be a surjective $(L, \gamma)$-homogeneously bi-H\"older map 
between metric spaces. 
Then for every $x\in X$ and for every $r\in (0, \infty)$ we have 
\[
f(B(x, r))\subset B(f(x), Lr^{\gamma})\subset f(B(x, L^{2/\gamma}r)).
\] 
\end{lem}

%%%%%%%%%%%%%%%%%%%%%%%%%%%%%%%
By Lemma \ref{lem:Holder}, 
we find that being a pre-tiling space is invariant under homogeneously bi-H\"older maps. 
\begin{prop}\label{prop:holder}
Every homogeneously bi-H\"older image of an arbitrary pre-tiling space is a pre-tiling space.
More precisely, 
the image of an arbitrary $(N, s)$-pre-tiling space 
under an $(L, \gamma)$-homogeneously bi-H\"older map is an $(N, s^{\gamma})$-pre-tiling space. 
\end{prop}

%%%%%%%%%%%%%%%%%%%%%%%%%%%%%%%
Since bi-Lipschitz maps are homogeneously bi-H\"older, 
 we have: 
\begin{cor}\label{cor:lip}
Every bi-Lipschitz image of an arbitrary $(N, s)$-pre-tiling space is
 an $(N, s)$-pre-tiling space.
\end{cor}
In spite of the virtue of Proposition \ref{prop:holder}, 
a homogeneously bi-H\"older image of a tiling space is 
not always a tiling space (see Example \ref{exam:seqCan}). 

%%%%%%%%%%%%%%%%%%%%%%%%%%%%%%%
By Proposition \ref{prop:snow}, 
we find that  a specific bi-H\"older image of a tiling space is a tiling space.  
\begin{prop}
Let $(X, P)$ be an $(N, s)$-tiling space and let $\epsilon\in (0, 1)$. 
Then $(X^{\epsilon}, P)$ is an $(N, s^{\epsilon})$-tiling space. 
\end{prop}

Let $X$ be a metric space and $P:\dom(P)\to \cov(X)$. 
Define a map $P^{C}:\dom(P)\to \cov(X)$ by $P^C_n=\{\, \cl(A)\mid A\in P_n\, \}$, 
where $\cl$ is the closure operator in $X$. 
The following proposition allows us to assume that tiles of pre-tiling spaces are closed sets. 
%%%%%%%%%%%%%%%%%%%%%%%%%
\begin{prop}\label{prop:closed}
Let $(X, P)$ be an $(N,s)$-pre-tiling space. 
Then $(X, P^C)$ is also an $(N,s)$-pre-tiling space. 
Moreover, 
if $(X, P)$ satisfies the condition \emph{(U)}, 
then so does $(X, P^C)$. 
\end{prop}
\begin{proof}
By the definition of $P^C$, the condition (S3) is satisfied.
From Lemma \ref{lem:nai} and   the condition (T2), 
it follows that for each pair   $n, m \in \dom(P)$ with $n<m$ and  for each $A\in P_n$, 
if $S, T\in [A]_{m-n}$ satisfy $S\neq T$, then $\cl(S)\neq \cl(T)$. 
Hence the condition (S1) is satisfied. 
By $\cl(A\cup B)=\cl(A)\cup \cl(B)$, 
the conditions (S2) is satisfied. 
Then the pair $(X, P^C)$ is an $N$-tiling set.  
By the facts that $\delta(A)=\delta(\cl(A))$ and that 
if $A\subset B$, then $\cl(A)\subset \cl(B)$, 
we conclude that $(X, P^C)$ is a pre-tiling space. 
Since for every subset  $A$ of $X$ we have  $d_{GH}(A, \cl(A))=0$, 
we obtain the latter part of the proposition. 
\end{proof}
%%%%%%%%%%%%%%%%%%%%%%%%
Let $X$ be a metric space. 
We say that a covering pair $(X, P)$ is \emph{self-similar} 
if there exists $s\in (0,1)$ such that for each $n\in \dom(P)$, 
for each $A\in P_n$ and for each $B\in P_{n+1}$, 
we have $d_{GH}(sA, B)=0$. 
By the definition of the self-similarity, 
we have:
\begin{lem}\label{lem:selfs}
Let $X$ be a metric space.
If a covering pair $(X, P)$ is self-similar, then
$(X,P)$ satisfies the conditions \emph{(T1)} and \emph{(U)}. 
\end{lem}
For a product of pre-tiling spaces, we obtain:
\begin{prop}
Let $p\in [1, \infty]$. 
Let $(X,P)$ and $(Y, Q)$ be $(N,s)$-pre-tiling spaces with  $\dom(P)=\dom(Q)$. 
Define a covering structure $R: \dom(P)\to \cov(X)$ by 
  $R_n=\{\, A\times B\mid A\in P_n,\ B\in Q_n\, \}$. 
Then  the covering pair $(X\times_pY, R)$ is a $(N^2, s)$-pre-tiling space. 
\end{prop}
\begin{proof}
%%%%%%%
Since $(X,P)$ and $(Y, Q)$  satisfy the conditions (S1), (S2) and (S3),  
%%%%%%%
so does $(X\times _pY, R)$.  
By the definition of the $\ell^p$-product metric, 
we  conclude that the conditions (T1) and (T2) are satisfied. 
Hence the covering pair $(X\times_pY, R)$ is a $(N^2, s)$-pre-tiling space. 
\end{proof}

\begin{rmk}
The author does not know whether it is true that if $X$ and $Y$ satisfy (U), 
then  so does $X\times_pY$ for any $p\in [1, \infty]$. 
\end{rmk}

%%%%%%%%%%%%%%%%%%%%%%%%%%%%%%%%%%
%%%%%%%%%%%%%%%%%%%%%%%%%%%%%%%%%%
%%%%%%%%%%%%%%%%%%%%%%%%%%%%%%%%%%
\section{Tiling spaces and the Assouad dimensions}\label{sec:Assouadtiling}
In this section, 
we prove Theorem \ref{thm:tiling}. 
\begin{prop}\label{prop:bdd}
Let $(X, P)$ be a doubling pre-tiling space. 
Then 
for every $W\in (0,\infty)$
there exists $M_W\in\nn_{\ge 1}$ such that 
for each $m\in \dom(P)$ and  for each subset  $S$ of $X$ with $\delta(S)\le Ws^m$, we have 
\[
\card(\{\, A\in P_m\mid A\cap S\neq \emptyset\, \})\le M_W.
\] 
\end{prop}
\begin{proof}
Let $D_2$ and $E$ be constants appearing  in the conditions (T1) and (T2).
Let $W\in (0, \infty)$, and take $m\in \dom(P)$ and  a subset $S$ of $X$ satisfying 
$\delta(S)\le Ws^m$. 
For each $A\in P_m$ with $A\cap S\neq \emptyset$, 
let $p_A\in A$ be a point appearing  in the condition (T2). 
Set 
\[
Z=\{\, p_A\in X\mid A\in P_m,\ A\cap S\neq \emptyset\, \}.
\] 
By  Lemma \ref{lem:nai} and the condition (T2), 
we have 
\[
\card(Z)=\card(\{\, A\in P_m\mid A\cap S\neq \emptyset\, \})
\]
 and $\alpha(Z)\ge Es^m$. 
By the condition (T2), we have $\delta(A)\le D_2s^m$ for every $A\in P_m$. 
For every point $x\in S$, we have $Z\subset B(x, (D_2+W)s^m)$. 
Thus we have $\delta(Z)\le 2(D_2+W)s^m$. 
By Proposition \ref{prop:Assouaddim}, 
we can take $\gamma\in \mathscr{C}(X)$. 
Then we have 
\begin{align*}
&\card(\{\, A\in P_m\mid A\cap S\neq \emptyset\, \}) 
\le C\left(\frac{\delta(Z)}{\alpha(Z)}\right)^{\gamma}
\le C\left(\frac{2(D_2+W)}{E}\right)^{\gamma}
\end{align*}
for some $C\in (0, \infty)$. 
This completes the proof. 
\end{proof}

\begin{df}\label{df:qqq}
Let $(X, P)$ be a tiling space, and let  $F$ be  a subset of $X$. 
For each pair $n, m\in \dom(P)$ with $n<m$ and for each $B\in P_n$,  
we put
\begin{equation}\label{eq:defQ}
\mathcal{Q}_{n,m}^{F}(B)=\{\, A\in P_m\mid A\cap F\neq \emptyset,\  A\subset B\, \}.
\end{equation}
\end{df}
%%%%%%%%%%MM%%%%%%%%
We show the following:
%%%%%%%%%%%%
%%%%%%%%%%%%%%%%%%%%%%%%%%%%%%%%%%
\begin{lem}\label{lem:dim}
Let $(X,P)$ be a doubling pre-tiling space. 
Let $F$ be a subset of $X$. 
Let $\Delta$ be the infimum of $\beta\in (0, \infty)$ 
for which  there exists $C\in (0, \infty)$ such that 
for each pair  $n, m\in \dim(P)$ with $n<m$ and for each $B\in P_n$ we have 
\begin{equation}\label{eq:qf}
\mathcal{Q}_{n,m}^{F}(B)\le C(s^{n-m})^{\beta}. 
\end{equation} 
Then we have 
\[
\dim_A F=\Delta.
\] 
\end{lem}
\begin{proof}
Let $(X, P)$ be a doubling $(N,s)$-pre-tiling space,  and let $D_1$ and $D_2$ be  constants appearing in the condition (T1). 
Take $\beta\in (0, \infty)$ satisfying (\ref{eq:qf}) with $\Delta<\beta$ . 
Let $S$ be a bounded subset of $F$. 
Take $n \in \dom(P)$ 
with $D_1s^{n-1}\le \delta(S)< D_1s^n$. 
Let $r\in (0, \infty)$, 
and take $m\in \dom(P)$ with $s^{m-1}\le r< s^m$. 
Applying  Proposition \ref{prop:bdd} to $D_1$,  
we obtain a constant $M_{D_1}$ stated in the proposition.  
Then $S$ can be covered by  at most $M_{D_1}$ members in $P_n$, 
and hence by (\ref{eq:qf})
the set $S$ can be covered 
by at most $M_{D_1}C(s^{n-m})^{\beta}$ members in $P_m$. 
In particular,  
we have
\[
\mathcal{B}_X(S, r)\le M_{D_1}C(s^{n-m})^{\beta}\le
 M_{D_1}D_1^{-\beta}Cs^{\beta}(\delta(S)/r)^{\beta}. 
\]
This implies $\beta\in \mathscr{A}(X)$. 
Hence, 
$\dim_AF\le \Delta$. 

We next prove the opposite inequality. 
Take $\beta\in \mathscr{A}(X)$ and  $B\in P_n$. 
The set $B\cap F$ is a bounded set of $X$ with $\delta(B\cap F)\le D_2s^n$. 
Thus $B\cap F$ can be covered by at most $C(s^{n-m})^{\beta}$  bounded sets 
with diameter at most $D_2s^{m}$. 
Write these bounded sets as $A_1, A_2, \dots, A_N$,  where $N\le C(s^{n-m})^{\beta}$. 
Applying  Proposition \ref{prop:bdd} to $D_2$, 
we obtain a constant $M_{D_2}$ stated in the proposition. 
Then each $A_i$ can be covered by at most  $M_{D_2}$ members in $P_m$.
Hence we have 
$\mathcal{Q}_{n,m}^F(B)\le M_{D_2}C(s^{n-m})^{\beta}. $
This implies $\Delta\le \dim_A(F)$. 
\end{proof}
%%%%%%%%%%%%%%%%%%%%%%%%%%%%%%%%%%

Applying Lemma \ref{lem:dim} to a whole pre-tiling space or  to a tile of it, 
we obtain the following: 
\begin{cor}\label{cor:dim}
Let $(X, P)$ be a doubling $(N, s)$-pre-tiling space. 
Then for every tile $T$ of $(X,P)$  we have 
\[
\dim_AT=\dim_AX=\log(N)/\log(s^{-1}). 
\]
\end{cor}

%%%%%%%%%%%%%%%%%%%%%%%%%%%%%%%%%%

%%%%%%%%%%%%%%%%%%%%%%%%%%%%%%%%%%
By the virtue of the condition (U), we obtain the following lemma:
\begin{lem}\label{lem:condiU}
Let $(X, P)$ be a doubling tiling space.  
Let $F$ be a subset of $X$. 
Let $D_2$ be a constant appearing  in the condition \emph{(T1)}. 
If $\tpc(F)$ contains no tiles of $(X, P)$, 
then 
there exists $k\in \nn$ such that 
   for each $n\in \dom(P)$ and for each $B\in P_n$ we have $d_{GH}(B, B\cap F)> D_2s^{n+k}$. 
\end{lem}
\begin{proof}
Suppose that for each $k\in \nn$ there exist $n_k\in \dom(P)$ and $B_k\in P_{n_k}$ such that 
$d_{GH}(B_k, B_k\cap F)\le  D_2s^{n_k+k}$. 
By the condition (T1),  we have 
\begin{equation}\label{eq:dgh}
d_{GH}(B_k, B_k\cap F)\le  (D_2/D_1)s^k\cdot \delta(B_k), 
\end{equation}
where $D_1$ is a constant appearing in the condition (T1). 
By the condition (U), there exists a subsequence $\{B_{\phi(k)}\}_{k\in \nn}$
of $\{B_k\}_{k\in \nn}$ such that $\delta(B_{\phi(k)})^{-1}B_{\phi(k)}$ converges to $(\delta(T))^{-1}T$ for some tile $T$ of  $(X, P)$. 
From (\ref{eq:dgh}) it follows that 
\[
d_{GH}((\delta(B_{\phi(k)}))^{-1}B_{\phi(k)}, (\delta(B_{\phi(k)}))^{-1}(B_{\phi(k)}\cap F))\le (D_2/D_1)s^{\phi(k)}. 
\]
By  $s^{\phi(k)}\to 0$ as $k\to \infty$,  we conclude that 
 $T\in \tpc(X)$. 
This is a contradiction. 
\end{proof}

%%%%%%%%%%%%%%%%%%%%%%%%%%%%%%%%%%
We next prove the following:
\begin{lem}\label{lem:dd}
Let $(X, P)$ be a doubling $(N, s)$-tiling space.  
Let $F$ be a subset of $X$. 
If $\tpc(F)$ contains no tiles of $(X, P)$, 
then we have  $\dim_{A}F<\log(N)/\log(s^{-1})$. 
\end{lem}
\begin{proof}
Let $D_1$ and $D_2$ be constants appearing  in the condition (T1). 
Set $d=\log(N)/\log(s^{-1})$. 
Take  $k\in \nn$ stated in Lemma \ref{lem:condiU}. 
Put $L=s^k$. 
By Lemma \ref{lem:condiU}, 
for each $n\in \dom(P)$ and  for each $B\in P_n$, 
we have $d_H(B, B\cap F; X)> D_2s^{n+k}$. 
Thus we can take a point $x\in B$ such that for every $y\in F$ 
we have  $d_X(x, y)>D_2s^{n+k}$. 
Take  $C\in [B]_k$ with $x\in C$, 
then by the condition (T1)  we have $C\cap F=\emptyset$. 
Therefore we obtain the following: 
\begin{clm}\label{sublem:cf}
For each $n\in \dom(P)$ and  for each $B\in P_n$, there exists $C\in P_{n+k}$ with $C\subset B$ and $C\cap F=\emptyset$. 
\end{clm}

Fix $a, b\in \dom(P)$ with $a>b$ and $B\in P_b$. 
Take $v\in \nn$ such that 
\begin{equation}\label{eq:defv}
D_1s^{b+k(v+1)}\le D_2s^a< D_1s^{b+kv}. 
\end{equation}
Since $D_2s^a< D_1s^{b+kv}$, 
we have $b+kv< a$.
Hence
for each $A\in P_{b+kv}$ 
the set $[A]_{a-(b+kv)}$ is non-empty. 
Let $W$ be the set of all words generated by $\{0, \dots, N^k-1\}$ 
whose length is at most $v$. Note that $W$ contains the empty word. 
For $w\in W$, we denote by $|w|$ the length of $w$. 
For $u, v\in W$, we denote by $uv$ the word product of $u$ and $v$. 

Let the set $\bigcup_{i=0}^v[B]_{ki}$ be indexed by $W$, say $\{T_w\}_{w\in W}$ such that 
for each $w\in W$ we have $T_w\in [B]_{k|w|}$, and such  that 
if $|w|<v-1$, then $T_{w0}\cap F=\emptyset$. 
This is possible by Sublemma \ref{sublem:cf}.

For each $i\in\{ 1, \dots, v\}$, define a set $H_i$ by
\[
H_i=\{\, w0\mid \text{$w\in W$, $|w|=i-1$ and all entries of $w$ are not $0$}\, \}. 
\]
Put $R_w=[T_w]_{a-(b+k|w|)}$. 
Remark that $R_w=\{\, A\in P_a\mid A\subset T_w\, \}$. 
Put $H=\bigcup_{i=1}^vH_i$. 
Note that for all distinct $v, w\in H$, the sets $R_v$ and $R_w$ are disjoint. 

Let 
$G=\bigcup_{w\in H}R_w$. 
We find that 
$ G=\coprod_{i=1}^v\coprod_{w\in H_i}R_w$. 
Let  $\mathcal{Q}_{b, a}^F(B)$ be the quantity  defined in Definition \ref{df:qqq}. Then we have
\[
\mathcal{Q}_{b, a}^F(B)\le \card([B]_{a-b})-\card(G).
\]
Since $d=\log(N)/\log(s^{-1})$, we have 
\begin{align*}
\card(R_w)&=N^{a-b-k|w|}=s^{-d(a-b-k|w|)},  \\
\card(H_i)&=(N^k-1)^{i-1}=(L^{-d}-1)^{i-1}. 
\end{align*}
By these equalities,  we obtain
\begin{align*}
\card(G) &=\card\left(\coprod_{i=1}^v\coprod_{w\in H_i}R_w\right)
=\sum_{i=1}^v\card\left(\coprod_{w\in H_i}R_w\right)\\
&=\sum_{i=1}^v\sum_{w\in H_i}s^{-d(a-b-k|w|)}= \sum_{i=1}^v\sum_{w\in H_i}s^{-d(a-b-ki)}\\
&=\sum_{i=1}^vs^{d(b-a)}s^{kdi}(L^{-d}-1)^{i-1}
=s^{d(b-a)}\sum_{i=1}^vL^{di}(L^{-d}-1)^{i-1}. 
\end{align*}
Since for each $w\in H$, we have $T_w\cap F=\emptyset$, 
by the definition (\ref{eq:defQ}) of $\mathcal{Q}_{b, a}^F(B)$, 
 we also obtain 
\[
\mathcal{Q}_{b, a}^F(B)\le N^{a-b}-\card(G)= (s^{b-a})^d\left(1-\sum_{i=1}^vL^{id}(L^{-d}-1)^{i-1}\right). 
\]
Note that we have 
\begin{align*}
&\sum_{i=1}^vL^{di}(L^{-d}-1)^{i-1}=(L^{-d}-1)^{-1}\sum_{i=1}^v(1-L^d)^i\\
&=(L^d-1)^{-1}(1-L^d)(1-(1-L^d)^v)L^{-d}=1-(1-L^d)^v. 
\end{align*}
By (\ref{eq:defv}), we have  $L^{v+1}\le (D_2/D_1)s^{a-b}$,  
then \[
\log((D_2/D_1)s^{a-b})/\log L-1\le v, 
\]
and hence 
\begin{align*}
\mathcal{Q}_{b, a}^F(B)&\le (s^{b-a})^d(1-\sum_{i=1}^vL^{di}(L^{-d}-1)^{i-1})=(s^{b-a})^d(1-L^d)^v\\
&\le (s^{b-a})^d(1-L^d)^{\log ((D_2/D_1)s^{a-b})/\log L-1}\\
&=(s^{b-a})^{d}\frac{1}{1-L^d}((D_1/D_2)s^{b-a})^{-\log(1-L^d)/\log L}\\
&=\frac{1}{1-L^d}\left(\frac{D_1}{D_2}\right)^{-\log(1-L^d)/\log L}(s^{b-a})^{d-\log(1-L^d)/\log L}. 
\end{align*} 
By Lemma \ref{lem:dim},  we obtain 
\[
\dim_A(F)\le d-\log(1-L^d)/\log L<d.
\] 
This completes the proof. 
\end{proof}

%%%%%%%%%%%%%%%%%%%%%%%%%%%%%%%%%%

\begin{lem}\label{lem:check}
Let $(X, P)$ be a pre-tiling space. 
Then the following are equivalent:
\begin{enumerate}
\item there exists a tile $A$ of $(X, P)$ such that $A\in \pc(X)$; 
\item there exists a tile of $A$ of  $(X, P)$ such that $F$ satisfies the asymptotic Steinhaus property for $A$. 
\end{enumerate}
\end{lem}
\begin{proof}
We first show that  $(2)\To (1)$. 
Take a tile $A$ of $(X, P)$ stated in (2). 
 For each $n\in \nn_{\ge 1}$, take a  $(1/n)$-net $S_n$ of $A$.
By the condition (2), 
we can take a finite subset $T_n$ of $F$ and $\delta_n\in (0, \infty)$ such that 
 $d_{GH}(T_n, \delta_n S_n)<\delta_n/n$. 
 Set $u_n=\delta_n^{-1}$, then  
 $ d_{GH}(u_nT_n, S_n)\le 1/n$. 
Hence we have 
$ \lim_{n\to \infty}d_{GH}(u_nT_n, A)=0$. 
 This implies $A\in \pc(X)$. 

We next show $(1)\To (2)$.  
Take a tile $A$ of $(X, P)$ stated in $(1)$. 
Take a finite subset $S$ of $A$. 
Since $A\in \pc(F)$, 
there exist a sequence $\{T_n\}_{n\in \nn}$ of subsets of $F$ 
and a sequence $\{u_n\}_{n\in \nn}$ in $(0, \infty)$ with 
$\lim_{n\to \infty}d_{GH}(u_nT_n, A)=0$. 
By Lemma \ref{lem:totbdd}, the tile  $A$ is totally bounded. 
Then, by Proposition \ref{prop:ttapprox},   for each $\epsilon\in (0, \infty)$, 
we can take a finite subset $Y_N$ of $T_N$ such that 
$d_{GH}(u_nY_N, A)<\epsilon$. 
Since $S$ is finite, we can take a subset $U_N$ of $Y_N$ such that 
 $d_{GH}(U_N, u_n^{-1}S)<u_n^{-1}\cdot \epsilon$. 
Thus $F$ satisfies the asymptotic Steinhaus property for $A$. 
\end{proof}

%%%%%%%%%%%%%%%%%%%%%%%%%%%%%%%%%%
In order to complete the proof of Theorem \ref{thm:tiling}, 
we recall the following theorem proved by the author \cite{I}. 
\begin{thm}\label{thm:ishiki}
Let $X$ be a metric space. If $P\in \pc(X)$, then we have 
\[
\dim_AP\le \dim_AX. 
\]
\end{thm}

\begin{proof}[Proof of Theorem \ref{thm:tiling}]
By the definitions, the implications (\ref{item:tpc})$ \To $ (\ref{item:pc}) 
and (\ref{item:kpc}) $\To$ (\ref{item:pc}) are true. 
Theorem \ref{thm:ishiki} implies that   (\ref{item:pc}) $\To $ (\ref{item:dimful}). 
Lemma \ref{lem:dd} is equivalent to (\ref{item:dimful}) $\To$ (\ref{item:tpc}). 
Lemma \ref{lem:check} states that (\ref{item:pc})$\iff$ (\ref{item:asp}) is true. 
Therefore it suffices to show that (\ref{item:pc}) $\To$ (\ref{item:kpc}). 

Let $(X, P)$ be a doubling tiling space.  
By Lemma \ref{lem:totbdd}, every tile of a tiling space is totally bounded. 
From this property and Proposition \ref{prop:ttapprox}, 
it follows that a tile of $(X, P)$ in $\pc(F)$ is approximated by 
a sequence of scalings of finite sets of $F$ in the sense of Gromov--Hausdorff. 
This completes the proof of  Theorem \ref{thm:tiling}. 
\end{proof}

%%%%%%%%%%%%%%%%%%%%%%%%%%%%%%%
%%%%%%%%%%%%%%%%%%%%%%%%%%%%%%%
%%%%%%%%%%%%%%%%%%%%%%%%%%%%%%%
%%%%%%%%%%%%%%%%%%%%%%%%%%%%%%%t

\section{Tiling spaces induced from  iterated function systems}\label{sec:exam}
%%%%%%%%%%%%%%%%%%%%%%%%%%%%%%%
%%%%%%%%%%%%%%%%%%%%%%%%%%%%%%
\subsection{Attractors} 
We first prove Theorem \ref{thm:frac}. 

Let $X$ be a complete metric space. 
Let $N\in \nn_{\ge 2}$ and $s\in (0,1)$. 
Let $\mathcal{S}=\{S_i\}_{i=0}^{N-1}$ be 
an $(N, s)$-similar iterated function system on $X$. 
Assume that the attractor $A_{\mathcal{S}}$ of $\mathcal{S}$ 
satisfies the strong open set condition. 
Let $V$ be an open set appearing  in the strong open set condition. 
Let $W$ be the set of all words generated by $\{0, \dots, N-1\}$. 
For every $q\in A_{\mathcal{S}}$, and for each $w\in W$, put $q_w=S_w(q)$. 

We first  prove that  $(A_{\mathcal{S}},P_\mathcal{S})$ is an $N$-tiling set, 
where $P_{\mathcal{S}}$  is the map defined  in Definition \ref{def:ifs}.

\begin{lem}\label{lem:frac1}
The covering pair $(A_{\mathcal{S}}, P_\mathcal{S})$ is an $N$-tiling set. 
\end{lem}
\begin{proof}
%%%%%%%%%%MM%%%%%%%
By the definition of $P_{\mathcal{S}}$, we see that the condition (S3) is satisfied. 
%%%%%%%%%%
We next verify that the condition (S1) is satisfied. 
Take $q\in A_{\mathcal{S}}\cap V$. 
For each pair  $n, m \in \nn$ with $n<m$, 
and for each $B\in P_n$, 
 by the definitions of the attractor and $P$, 
 we have $B=\bigcup[B]_{m-n}$. 
 
We show  $\card([B]_{m-n})=N^{m-n}$. 
By the definition of $P_{\mathcal{S}}$, 
we have 
\[
[B]_{m-n}=\{\, S_{w}(B)\mid \text{$w\in W$ and $|w|=m-n$}\, \}. 
\]
Write $B=S_{v}(A_{\mathcal{S}})$, where $v\in W$. 
By the condition (O2) in the strong open set condition, 
the family $\{\, S_{vw}(V)\mid \text{$w\in W$ and $|w|=m-n$} \, \}$ is pairwise disjoint.  
This implies that if $w\neq w'$ with $|w|=|w'|=m-n$, 
then $q_{vw}\neq q_{vw'}$. 
Hence 
the set $\{\, q_{vw}\in A_{\mathcal{S}}\mid \text{$w\in W$ and $|w|=m-n$} \, \}$ 
consists of $N^{m-n}$  elements. 
Since for each $w\in W$ with $|w|=m-n$ 
we have  $q_{vw}\in S_{vw}(V)\cap S_{vw}(A_{\mathcal{S}})$, 
  we obtain $\card([B]_{m-n})=N^{m-n}$. 
  
  By the boundedness of $A_{\mathcal{S}}$, 
  the pair $(A_{\mathcal{S}}, P_{\mathcal{S}})$  satisfies  the condition (S2). 
  Therefore the pair $(A_{\mathcal{S}}, P_{\mathcal{S}})$ is an $N$-tiling set.  
\end{proof}

We next prove that $(A_{\mathcal{S}}, P_{\mathcal{S}})$ is 
an $(N, s)$-tiling space. 

\begin{lem}\label{lem:hodai}
The attractor $A_{\mathcal{S}}$ of $\mathcal{S}$ is contained in  
$ \cl(V)$. Moreover, for every $w\in W$, 
we have $S_w(A_{\mathcal{S}})\subset \cl(S_w(V))$. 
\end{lem}
\begin{proof}
Take $q\in A_{\mathcal{S}}\cap V$.
Put $M_0=\{q\}$,  and for each $n\in \nn_{\ge 1}$ put
$M_n=\bigcup_{i=0}^{N-1}S_i(M_{n-1})$. 
Then $M_n$ converges to $A_{\mathcal{S}}$ in the Hausdorff topology, 
in particular, $\cl(\bigcup_{n\in \nn}M_n)=A_{\mathcal{S}}$. 
By the definition, for each $n\in \nn$ we have $M_n\subset V$. 
Thus $A_{\mathcal{S}}\subset \cl(V)$. 
Since $S_w$ is a topological embedding for any $w\in W$, 
the latter part follows from the former one. 
\end{proof}

\begin{proof}[Proof of Theorem \ref{thm:frac}]
Since $(A_{\mathcal{S}}, P_{\mathcal{S}})$ is self-similar, 
by Lemma \ref{lem:selfs} the covering  pair $(A_{\mathcal{S}}, P_{\mathcal{S}})$ satisfies the conditions (T1) and (U). 
It suffices to show that  $(A_{\mathcal{S}}, P_{\mathcal{S}})$ satisfies the condition (T2). 
By $A_{\mathcal{S}}\cap V\neq \emptyset$, 
we can take  $q\in A_{\mathcal{S}}\cap V$. 
Then there exists $E\in (0, \infty)$ such that $U(q, E)\subset V$. 
Since $\mathcal{S}$ consists of $s$-similar transformations, 
for each $w\in W$ 
we have $B(q_w, Es^{|w|})\subset S_w(V)$. 
By Lemma \ref{lem:hodai},  
for each $w\in W$ we have $S_w(A_{\mathcal{S}})\subset \cl(S_w(V))$, 
thus
the ball $B(q_w, Es^{|w|})$ in $X$ meets only $S_w(A_{\mathcal{S}})$. 
Hence
the ball $B(q_w, Es^{|w|})$ in $A_{\mathcal{S}}$ is a subset of 
$S_{w}(A_{\mathcal{S}})$. 
Therefore we conclude that $(A_{\mathcal{S}}, P_{\mathcal{S}})$ satisfies the condition (T2).  
This completes the proof of Theorem \ref{thm:frac}. 
\end{proof}

%%%%%%%%%%%%%%%%%%%%%%MODIFY%%%%%%%%%%%%%%%%
\begin{rmk}
An iterated function system  $\mathcal{S}$ on a metric space is said to satisfy the \emph{open set condition}
if $\mathcal{S}$ satisfies the conditions (O1) and (O2) in the strong open set condition. 
Schief \cite{Sch1} proved  that the open set condition implies the strong open set condition in the Euclidean setting.
Schief \cite{Sch2} also proved  that the implication mentioned above does not hold in a general setting (see \cite[Example 3.1]{Sch2}). 
\end{rmk}

%%%%%%%%%%%%%%%%%%%%%%MODIFY%%%%%%%%%%%%%%%%

%%%%%%%%%%%%%%%%%%%%%%%%%%%%%%
\subsection{Extended attractors}
We can  construct an unbounded tiling space induced from a similar iterated function system. 
\begin{df}[Extended attractor]\label{def:extended}
Let $N\in \nn_{\ge 2}$ and $s\in (0,1)$. 
Let $\mathcal{S}$ be an $(N, s)$-similar iterated function system on a complete metric space 
with the strong open set condition, 
say $\mathcal{S}=\{S_i\}_{i=0}^{N-1}$. 
Define a sequence $\{F_k\}_{k\in \nn}$ of metric spaces by $F_k=s^{-k}A_{\mathcal{S}}$, 
where $A_{\mathcal{S}}$ is the attractor of $\mathcal{S}$. 
Note that for each $k\in \nn$, 
each $S_i$ is an $s$-similar transformation on $F_k$. 
By the definition of $A_{\mathcal{S}}$, 
we find that $S_0(F_{k+1})$ is isometric to $F_{k}$. 
Thus we can identify $F_{k}$ with $S_0(F_{k+1})$,  
and we may consider that  $F_{k}\subset F_{k+1}$ for each $k\in \nn$. 
Put $E_{\mathcal{S}}=\bigcup_{k\in \nn}F_k$. 
Note that $E_{\mathcal{S}}$ is unbounded. 
Let $W$ be  the set of all words generated by $\{0, \dots, N-1\}$. 
Define a map $Q_{\mathcal{S}}:\zz\to \cov(E_{\mathcal{S}})$ by 
\begin{equation}\label{eq:extended}
(Q_{\mathcal{S}})_n=\{\, S_w(F_k)\mid\text{ $w\in W$ and $|w|-k=n$}\, \}. 
\end{equation}
We call  $E_{\mathcal{S}}$ an \emph{extended attractor of $\mathcal{S}$}. 
\end{df}
Similarly to Theorem \ref{thm:frac}, we obtain the following:
\begin{thm}\label{thm:extended}
For $N\in \nn_{\ge 2}$ and $s\in (0, \infty)$, 
let $\mathcal{S}$ be an $(N, s)$-similar iterated function system on a complete metric space with the strong open set condition. 
Let $E_{\mathcal{S}}$ be the extended attractor of $\mathcal{S}$, 
and 
$Q_{\mathcal{S}}$ the map defined by \emph{(\ref{eq:extended})}. 
Then $(E_{\mathcal{S}}, Q_{\mathcal{S}})$ is an unbounded $(N, s)$-tiling space. 
\end{thm} 
\begin{proof}
Since all the tiles of $(E_{\mathcal{S}}, Q_{\mathcal{S}})$ are similar to $A_{\mathcal{S}}$, 
by a similar argument to Lemma \ref{lem:frac1}, 
we see that the condition (S1) is satisfied. 
Lemma \ref{lem:selfs} implies that  $(E_{\mathcal{S}}, Q_{\mathcal{S}})$ satisfies the conditions (T1) and (U). 
By $E_{\mathcal{S}}=\bigcup_{k\in \nn}F_k$, 
and by the definition of $Q_{\mathcal{S}}$, 
%%%%%%%%%%%%MM%%%%%%
the conditions (S2) and (S3) are satisfied. 
Thus $(E_{\mathcal{S}}, Q_{\mathcal{S}})$ is an $N$-tiling set.
Similarly  to the proof of Theorem \ref{thm:frac}, 
we see that the condition (T2) is satisfied. 
Therefore the pair $(E_{\mathcal{S}}, Q_{\mathcal{S}})$ is an $(N, s)$-tiling space. 
\end{proof}

\subsection{Examples of attractors}\label{subs:attractor}
\begin{exam}[The middle-third Cantor set]
Let $C$ be the middle-third Cantor set. 
For each $i\in \{0, 1\}$,  define a map $f_i:\rr\to \rr$ by 
\[
f_i(x)=\frac{1}{3}x+\frac{2}{3}i. 
\]
Put $\mathcal{S}=\{f_0, f_1\}$. 
Then $\mathcal{S}$ is a $(2, 3^{-1})$-similar iterated function system on $\rr$,  
and $C$ is the attractor of $\mathcal{S}$. 
The open set $(0, 1)$ satisfies the conditions (O1), (O2) and (O3), 
and hence $\mathcal{S}$ satisfies the strong open set condition. 
Let $P_{\mathcal{S}}:\nn\to C$ be the map defined in Definition \ref{def:ifs}. 
Theorem \ref{thm:frac} implies that $(C, P_{\mathcal{S}})$ is a $(2, 3^{-1})$-tiling space. 
\end{exam}

\begin{exam}[The Sierpi\'nski gasket]
Referring to the cubic roots of unity,  
put 
$w_0=(1, 0)$, 
$w_1=2^{-1}(-1, \sqrt{3})$ and $w_2=2^{-1}(-1, -\sqrt{3})$. 
For each $i\in \nn$, we define a map $f_i:\rr^2\to \rr^2$ by 
\[
f_i(x)=\frac{1}{2}x+\frac{1}{2}w_i. 
\]
Put $\mathcal{S}=\{f_0, f_1, f_2\}$. Then $\mathcal{S}$ is a $(3, 2^{-1})$-similar iterated function sytem on $\rr^2$. 
The attractor $A_{\mathcal{S}}$ of $\mathcal{S}$ is called the Sierpi\'nski gasket. 
The interior $V$ of the triangle with vertices $\{w_0, w_1, w_2\}$ satisfies the conditions (O1), (O2) and (O3). 
Thus $\mathcal{S}$ satisfies the strong open set condition. 
Let $P_{\mathcal{S}}:\nn\to \cov (A_{\mathcal{S}})$ 
 be the map defined in Definition \ref{def:ifs}. 
 Then Theorem \ref{thm:frac} implies that 
 $(A_{\mathcal{S}}, P_{\mathcal{S}})$ is a $(3, 2^{-1})$-tiling space. 
\end{exam}

\begin{exam}[Euclidean spaces]\label{exam:Eu}
Consider the $N$-dimensional normed vector space $\rr^N$ with 
 $\ell^p$-metric, where $p\in [1, \infty]$.  

Let $A=\{\, v\in \rr^N\mid \text{the entries of $v$ are $0$,  $1$ or $-1$}\, \}$. 
Since $A$ has cardinality $3^N$,
it is indexed by  $\{1, \dots, 3^N\}$, say $A=\{v(i)\}_{i=1}^{3^N}$. 
For each $i\in \{1,\dots, 3^N \}$, 
define a $(1/3)$-similar transformation $f_i:\rr^N\to \rr^N$ by 
\[
f_i(x)=\frac{1}{3}x+\frac{1}{3}v(i). 
\]
Put $\mathcal{S}=\{f_i\}_{i=1}^{3^N}$. 
Then $\mathcal{S}$ is a $(3^{N}, 3^{-1})$-similar iterated function system on $\rr^{N}$, and  $[-2^{-1},2^{-1}]^N$ is the attractor of $\mathcal{S}$. 
The open set $(-2^{-1},2^{-1})^{N}$ satisfies the conditions (O1), (O2) and (O3). 
Hence $\mathcal{S}$ satisfies the strong open set condition. 
Let  $P_{\mathcal{S}}: \nn\to \cov ([-2^{-1},2^{-1}]^N)$ be the map defined in Definition \ref{def:ifs}, 
then this map is described as  
\[
(P_{\mathcal{S}})_n=\{\, 3^{-n}v+3^{-n}[-2^{-1},2^{-1}]^N\mid v\in \zz^N\, \}. 
\]
Theorem \ref{thm:frac} implies  that 
$([-2^{-1},2^{-1}]^N, P_{\mathcal{S}})$ is a $(3^N, 3^{-1})$-tiling space. 

We next consider the extended attractor $E_{\mathcal{S}}$ of $\mathcal{S}$. 
Since 
\[
\rr^{N}=\bigcup_{i\in \nn}[-2^{-1}\cdot 3^i, 2^{-1}\cdot 3^i]^N,
\] 
the space $E_{\mathcal{S}}$ is isomeric to $\rr^N$ with $\ell^p$-metric. 
Under a natural identification between $E_{\mathcal{S}}$ and $\rr^N$, 
the map $Q_{\mathcal{S}}:\zz\to \cov(\rr^N)$ defined in Definition \ref{def:extended} is described as 
\[
(Q_{\mathcal{S}})_n=\{\, 3^{-n}v+3^{-n}[-2^{-1},2^{-1}]^N\mid v\in \zz^N\, \}. 
\]
Theorem  \ref{thm:extended} implies that  $(\rr^N, Q_{\mathcal{S}})$ is a $(3^N, 3^{-1})$-tiling space. 
\end{exam}
 Applying Theorem \ref{thm:tiling} to the tiling space $(\rr^N, P_{\mathcal{S}})$ discussed in Example \ref{exam:Eu}, 
we obtain the  Fraser--Yu  characterization in \cite{FY} in a slightly different formulation:
\begin{cor}
For every subset $F$ of $\rr^N$, the following are equivalent:
\begin{enumerate}
\item $\dim_AF=\dim_A\rr^N$; \label{item:dimful}
\item $[0,1]^N\in \pc(F)$; \label{item:pc}
\item $[0,1]^N\in \tpc(F)$; \label{item:tpc}
\item $[0,1]^N\in \kpc(F)$;  \label{item:kpc}
\item $F$ satisfies the asymptotic Steinhaus property for $[0,1]^N$. \label{item:asp}
\end{enumerate}
\end{cor}
\begin{proof}
Let $P_{\mathcal{S}}$ be the map described in Example \ref{exam:Eu}. 
Since all the tiles of $(\rr^N, P_{\mathcal{S}})$ are similar to $[0, 1]^N$, 
Theorem \ref{thm:tiling} leads to  the claim. 
\end{proof}

\begin{exam}[$p$-adic numbers]\label{exam:padic}
Let $p$ be a prime number and let $v_p$ be the $p$-adic valuation. 
Let $\qq_p$ be the set of all $p$-adic numbers. 
Let $s\in (0,1)$, and  define  $d_{\qq_p}(x,y)=s^{v_p(x-y)}$, then 
$d_{\qq_p}$ is an ultrametric on $\qq_p$. 
For each $k\in \{0, \dots, p-1\}$, 
define an $s$-similar transformation
$f_k:\qq_p\to \qq_p$ by 
\[
f_k(x)=xp+k. 
\] 
Put $\mathcal{S}=\{f_k\}_{k=0}^{p-1}$. 
Then $\mathcal{S}$ is a $(p, s)$-similar iterated function system on $\qq_p$. 
The ball $B(0, 1)$ centered at $0$ with radius $1$ in $\qq_p$ is the attractor of $\mathcal{S}$. 
Since for each $k\in \{0, \dots, p-1\}$ we have $f_k(B(0, 1))=B(k, s)$, 
and  since $d_{\qq_p}$ is an ultrametric, 
the open set $B(0, 1)$ satisfies the conditions (O1), (O2) and (O3). 
Thus $\mathcal{S}$ satisfies the strong open set condition. 
Let  $P_{\mathcal{S}}:\nn\to \cov(B(0,1))$ be the  map defined in Definition \ref{def:ifs}, 
then this map is described as
\[
(P_{\mathcal{S}})_n=\{\, B(a,s^{-n})\mid a\in \qq_p\, \}. 
\]
By Theorem \ref{thm:frac}, 
we conclude that $(B(0,1), P_{\mathcal{S}})$ is a $(p,s)$-tiling space. 

We next consider the extended attractor of $\mathcal{S}$. 
Since 
 \[
 \qq_p=\bigcup_{i\in \nn}B(0, s^{-i}),
 \]
  the space $E_{\mathcal{S}}$ is 
isometric to $\qq_p$. Under a natural identification between $E_{\mathcal{S}}$ and $\qq_p$, 
the map $Q_\mathcal{S}:\zz\to \cov(\qq_p)$ defined in Definition \ref{def:extended} is described as
\[
(Q_\mathcal{S})_n=\{\, B(a,s^{-n})\mid a\in \qq_p\, \}.
\] 
Theorem \ref{thm:extended} implies that $(\qq_p, Q_{\mathcal{S}})$ is 
a $(p, s)$-tiling space. 
By Corollary \ref{cor:dim}, we obtain $\dim_A\qq_p=\log(p)/\log(s^{-1})$. 
\end{exam}

%%%%%%%%%%%%%%Modify%%%%%%%%%%%%%%%
%%%%%%%%%%%%%%%%%%%%%%%%%%%
%%%%%%%%%%%%%%%%%%%%%%%%%%%%%%%

%%%%%%%%%%%%%%%%%%%%%%%%%%%%%%%
\section{Counterexaples}\label{sec:counter}
We first provide a tiling space that is not doubling. 
\begin{exam}\label{exam:nondb}
Let $N\in \nn_{\ge 2}$ and $s\in (0, 1)$. 
Let $T$  be the set of all sequences $x:\nn\to \{0, \dots, N-1\}$ 
 satisfying 
that $x_0\in \{0, \dots, N-2\}$.  
The set $T$ can be described as 
\[
T=\{0, \dots, N-2\}\times \prod_{n=1}^{\infty }\{0, \dots, N-1\}. 
\]
For $x, y\in T$,  define a valuation $v:T\times T\to \nn\cup\{\infty\}$ by 
\[
v(x, y)=
\begin{cases}
\min\{\, n\in \nn\mid x_n\neq y_n\, \} & \text{if $x\neq y$}, \\
\infty & \text{if $x=y$}.
\end{cases}
\]
For each $i\in \nn$, 
let $T_i$ be  the metric space $(T, d_i)$, 
where the metric $d_i$ is defined by $d_i(x,y)=s^{v(x,y)+i}$. 

For each $i\in \nn$, the symbol $o_i$ stands for 
the sequence whose all entries are $0$ in $T_i$. 
For each $k\in \nn$,  put $O_k=\coprod_{i=k}\{o_i\}$,
and 
\[
X(k)=\left(\coprod_{i\in \nn_{\ge k}}T_i\right)/O(k). 
\]
Namely, 
$X(k)$ is the set constructed by identifying the zero sequences in the set $\coprod_{i\in \nn_{\ge k}}T_i$. 
Put $X=X(0)$. 
We may consider that for each $k\in \nn$ we have
$X(k+1)\subset X(k)$ and $T_k\subset X$. 
The symbol $o$ stands for the zero sequence in $X$.  
The point $o$ is the identified point in $X$.  

We define a function $d_X:X\times X\to \rr_{\ge 0}$  by 
\[
d_X(x, y)=
	\begin{cases}
	d_i(x,y) & \text{if $x, y\in T_i$ for some $i$, }\\
	d_i(x,o)+d_j(o,y) & \text{if $x\in T_i$ and $y\in T_j$ for some $i\neq j$. }
	\end{cases}
\]
The function $d_X$ is a metric on $X$. 

We next define a tiling structure on $X$. 
Let $W$ be the set of all words whose $0$-th entries are in $\{0,\dots, N-2\}$ and other entries are in $\{0,\dots, N-1\}$.  
Remark that the set $W$ does not contain the empty word.
For each word $w=w_0\cdots w_l\in W$, 
we define 
\[
(T_i)_w=\{\, x\in T_i\mid x_0=w_0, \dots, x_l=w_l\, \}, 
\] 
where $w_0\in \{0,\dots,  N-2\}$ and $w_i\in \{0, \dots, N-1\}$ ($i\ge 1$). 
For each $k\in \nn$ and for each $l\in \nn_{\ge 1}$, 
we define 
\[
\mathcal{S}_{k,l}=\{\,(T_k)_w\mid \text{$w\in W$ and $|w|=l$}\, \}. 
\]
We define a map  $P:\nn\to \cov(X)$  by
\[
P_n=\{X(n)\}\cup\bigcup_{k+l=n}\mathcal{S}_{k,l}. 
\]

We first show that $(X, P)$ is an $N$-tiling set. 
%%%%%%%%%
By the definition of $P$, the condition (S3) is satisfied. 
%%%%%%%
For each $w\in W$, we have $(T_i)_w=\bigcup_{v=0}^{N-1}(T_i)_{wv}$. 
For each $n\in \nn$, we have 
\[
X(n)=X(n+1)\cup \bigcup_{v=0}^{N-2}(T_n)_v.
\] 
%%%%%%%%%%%%%%%%%
Thus, the conditions (S1) is satisfied. 
%%%%%%%%%%%%%%%%%%
By the boundedness of $X$, the condition (S2) is satisfied. 
Thus, the pair $(X, P)$ is an $N$-tiling set. 

We next show that $(X, P)$ is an $(N, s)$-tiling space. 
By the definition of the metric $d_X$, 
for each $n\in \nn$, and for all $k\in \nn$ and $ l\in \nn_{\ge 1}$ with $k+l=n$, 
for each $(T_i)_w\in \mathcal{S}_{k, l}$, 
we have $\delta((T_i)_w)=s^{-n}$ and $\delta(X(n))=s^{-n}+s^{-n-1}$. 
By $s^{-n}\le s^{-n}+s^{-n-1}\le 2s^{-n}$, 
the condition (T1) is satisfied. 
By the definition of $\{T_i\}_{i\in \nn}$,  
for every $a\in T_i$ we have  $(T_i)_w=B(a, s^{-n})$. 
For every $a\in (T_n)_1$ we also have 
\[
B(a, s^{-n-1})\subset T_n\subset X(n).
\] 
Then the condition (T2) is satisfied. 
For each $n\in \nn$, the spaces $sX(n)$ and $X(n+1)$ are isometric to each other. 
For all $i, j\in \nn$ and for all $u,v\in W$, 
the spaces $(T_i)_u$ and $(T_j)_v$ are similar. 
Thus the tiles of $(X, P)$ have two similarity classes, 
and hence the condition (U) is satisfied. 
Therefore $(X, P)$ is an $(N, s)$-tiling space. 

%%%%%%%%%%%%%%%MMMM%%%%%%%
For each $n\in \nn$, 
we have $\card(\{\, A\in P_n\mid o\in A\, \})=(n-1)(N-1)+N$. 
By Proposition \ref{prop:bdd}, 
and by $(n-1)(N-1)+N\to \infty$ as $n\to \infty$, 
we conclude that  $X$ is not doubling. 
%%%%%%%%%%%%%%%%%%%%%%%%%%%%%
\end{exam}
\begin{rmk}
Due to the Brouwer characterization of the middle-third Cantor set, 
the space $X$ constructed in Example \ref{exam:nondb} is homeomorphic to the middle-third Cantor set. 
Indeed, 
the space $X$ is topologically $0$-dimensional and compact, 
and  it has no isolated points. 
\end{rmk}
\begin{rmk}
In  Example \ref{exam:nondb}, 
by replacing the role of $\nn$ with that of $\zz$, 
we also obtain an unbounded non-doubling  $(N, s)$-tiling space that is not locally compact. 
Therefore being a tiling space does not imply the local compactness. 
\end{rmk}
%%%%%%%%%%%%%%%%%%%%%%%%%%%%
%%%%%%%%%%%%%%%%%%%%%%%%%%%%

%%%%%%%%%%%%%%%%%%%%%%%%%%%%%%%
We next  construct a pre-tiling space that is not a tiling space. 
The space constructed below is also a bi-Lipschitz image of a tiling space. 
\begin{exam}\label{exam:seqCan}
Let $2^{\nn}$ be the set of all sequences valued in $\{0,1\}$. 
For $x, y\in 2^{\nn}$, 
define a valuation $v:2^{\nn}\times 2^{\nn}\to \nn\cup\{\infty\}$ by 
\[
v(x, y)=
\begin{cases}
\min\{\, n\in \nn\mid x_n\neq y_n\, \} & \text{if $x\neq y$}, \\
\infty & \text{if $x=y$}.
\end{cases}
\]
For $n\in \nn$,  set $a_n=(1-1/(n+3))2^{-(n+3)}$. 
Let  $X$ be  a metric space $(2^{\nn}, d_X)$ with  metric $d_X$ defined by $d_X(x, y)=a_{v(x, y)}$. 
Let $Y$ be a metric space $(2^{\nn}, d_Y)$ with metric $d_Y$ defined by  $d_Y(x, y)=2^{-v(x, y)}$. 
The spaces $X$ and $Y$ are ultrametric spaces. 
Define two maps $P, Q:\nn\to \cov(2^{\nn})$ by
$P_n=\{B(x, a_n)\mid x\in X\}$
 and 
$ Q_n=\{B(x, 2^{-n})\mid x\in Y\}$. 
Then $(X, P)$ and $(Y, Q)$ are $2$-tiling sets. 

We now prove that $(Y, Q)$ is a $(2, 2^{-1})$-tiling space. 
For each  $i\in \{0,1\}$, define a mep $f_i:2^{\nn}\to 2^{\nn}$ by 
\[
(f_i(x))_j=
\begin{cases}
i & \text{if $j=0$, }\\
x_{j-1} &\text{if $j\ge 1$}, 
\end{cases}
\]
where $(f_i(x))_j$ is the $j$-th entry of $f_i(x)$. 
Then $\{f_0, f_1\}$ is a $(2, 2^{-1})$-similar iterated function system on Y, 
and 
$Y$ is the attractor of $\{f_0, f_1\}$. 
 The map $P_{\{f_0, f_1\}}:\nn\to \cov(2^{\nn})$ coincides with the map 
 $Q$. 
Therefore  Theorem \ref{thm:frac} implies that
the  pair $(Y, Q)$ is a $(2, 2^{-1})$-tiling space. 

Note that the tiles of  $(X, P)$ have  infinitely many similarity classes. 
 The similarity classes of the tiles of $(X, P)$ do not contain that of $(Y, Q)$. 

The identity map $id:X\to Y$ is bi-Lipschitz,  
in particular,  
the metric space $X$ is a bi-Lipschitz image of $Y$. 
Since $(Y, Q)$ is a $(2, 2^{-1})$-pre-tiling space, 
by Corollary \ref{cor:lip},  so is  $(X, P)$.

Take a sequence $\{A_i\}_{i\in \nn}$ of tiles of $(X, P)$ with $A_i\in P_i$. 
For each $N\in \nn$ and for each  $n\in \nn$,  
we have $|a_{n+N}/a_N-2^{-n}|<1/(N+2)$. 
Then the sequence $(\delta(A_i))^{-1}A_i$ converges to $Y$ 
in the sense of Gromov--Hausdorff.  
Thus $(X, P)$ is a pre-tiling space which does not satisfy the condition (U).
 
In summary, 
the pre-tiling space $(X, P)$ is a non-tiling  space which is a bi-Lipschitz image of the tiling space $(Y, Q)$. 
\end{exam}
%%%%%%%%%%%%%%%%%%%%%%%%%%%%%%%
Combining  the metric spaces provided in Example \ref{exam:seqCan},  
we construct a tiling space whose tiles  have infinitely many similarity classes.  
\begin{exam}\label{exam:infinite}
Let $(X, P)$ and $(Y, Q)$ be the pre-tiling space and the tiling space constructed in  Example \ref{exam:seqCan}, 
respectively. 
Put $Z=X\sqcup Y$ and define a metric $d_Z$ on $Z$ by 
\[
d_Z(x,y)=
\begin{cases}
d_X(x,y) &\text{if $x,y\in X$},\\
d_Y(x,y) &\text{if $x,y\in Y$},\\
2 & \text{if $x$ and $y$ lie in distinct components}. 
\end{cases}
\]
We now define a map $R:\nn\to \cov(X\sqcup Y)$ by 
$R_0=\{Z\}$,  and by  $R_n=P_{n-1}\cup Q_{n-1}$ for $n\in \nn_{\ge 1}$. 
Since $(X, P)$ and $(Y, Q)$ are $(2, 2^{-1})$-pre-tiling spaces,  
the pair $(Z, R)$ is a $(2, 2^{-1})$-pre-tiling space. 

We now prove that $(Z, R)$ satisfies the condition (U). 
Take a sequence $\{A_n\}_{n\in \nn}$ of tiles of $(Z,R)$. 
Then there exists a subsequence $\{A_{n_i}\}_{i\in \nn}$ 
of $\{A_n\}_{n\in \nn}$ consisting of  
either tiles of $(X, P)$ or that of $(Y, Q)$.  
If $\{A_{n_i}\}_{i\in \nn}$ consists of tiles of $(X, P)$, 
then by the argument in Example \ref{exam:seqCan}, 
 there exists a subsequence $\{A_{m_i}\}_{i\in \nn}$ of $\{A_{n_i}\}_{i\in \nn}$ such that $\delta(A_{m_i})^{-1}A_{m_i}$ converges to 
either $Y$ or $\delta(T)^{-1}T$ for some tile $T$ of $(X, P)$. 
In the case where $\{A_{n_i}\}_{i\in \nn}$ consists of tiles of $(Y, Q)$, 
since $(Y,Q)$ satisfies the condition (U), 
there exists a subsequence $\{A_{m_i}\}_{i\in \nn}$ of $\{A_{n_i}\}_{i\in \nn}$ such that $\delta(A_{m_i})^{-1} A_{m_i}$ converges to 
$\delta(T)^{-1}T$ for some tile $T$ of $(Y, Q)$. 
Therefore $(Z, R)$ satisfies the condition (U). 

In this way, we obtain a $(2,  2^{-1})$-tiling space  $(Z, R)$ whose tiles have infinitely many similarity classes. 
\end{exam}

%%%%%%%%%%%%%%%%%%%%%%%%%%%%%%%%%%%%%%%%%%%%%%%%%%%%%%%%%%%%%%%%%%%%%%%%%%%%%%%%%%%%%%%%%%%%%

%%%%%%%%%%%%%%%%%%%%%%%%%%%%%%%%%%%%%%%%%5
%%%%%%%%5

\end{document}